\documentclass[a4paper, reqno, 11pt]{amsart}
\usepackage{amsmath, amssymb, graphics, epsfig, enumerate, amstext, amsthm, comment}
\usepackage[noadjust]{cite}
\usepackage[plainpages=false,colorlinks,hyperindex,pdfpagemode=None,bookmarksopen,linkcolor=blue,citecolor=blue,urlcolor=blue]{hyperref}

\numberwithin{equation}{section}

\newtheorem{theorem}{Theorem}[section]

\newtheorem{corollary}[theorem]{Corollary}
\newtheorem{lemma}[theorem]{Lemma}

\theoremstyle{definition}

\newtheorem{example}[theorem]{Example}
\newtheorem{remark}[theorem]{Remark}

\theoremstyle{remark}
\newtheorem{problem}[theorem]{Problem}

\def\qed{\hfill\vbox{\hrule width 6 pt
\hbox{\vrule height 6 pt width 6 pt}} \medskip}

\def\IF{{\mathbb F}}

\def\IZ{{\mathbb Z}}

\def\bL{{\bf L}}
\def\bM{{\bf M}}

\def\bV{{\bf V}}

\def\b0{{\bf 0}}

\def\bV{{\bf V}}

\def\[{\left [}
\def\]{\right ]}
\def\({\left (}
\def\){\right )}

\def\<{{\langle}}
\def\>{{\rangle}}
\def\1{{\bf 1}}

\newcommand{\T}{{\thinspace\mathrm{t}}}

\newcommand{\bprob}{\begin{problem}}
\newcommand{\eprob}{\end{problem}}

\newcounter{example}
\protected\def\verythinspace{%
  \ifmmode
    \mskip0.33\thinmuskip
  \else
    \ifhmode
      \kern0.08334em
    \fi
  \fi
}

\begin{document}
\openup .75\jot
\title[Linear maps preserving  matrices annihilated by a fixed  polynomial]
{Linear maps 
preserving  matrices annihilated by a fixed  polynomial}

\author[Li, Tsai, Wang and Wong]{Chi-Kwong Li, Ming-Cheng Tsai, Ya-Shu Wang \and Ngai-Ching Wong}

\address[Li]{Department of Mathematics, The College of William
\& Mary, Williamsburg, VA 13185, USA.}
\email{ckli@math.wm.edu}

\address[Tsai]{General Education Center, National Taipei University of Technology, Taipei 10608, Taiwan.}
\email{mctsai2@mail.ntut.edu.tw}

\address[Wang]{Department of Applied Mathematics, National Chung Hsing University, Taichung 40227, Taiwan.}
\email{yashu@nchu.edu.tw}

\address[Wong]{Department of Applied Mathematics, National Sun Yat-sen
  University, Kaohsiung, 80424, Taiwan; Department of Healthcare Administration and Medical Information, and Center of
Fundamental Science, Kaohsiung  Medical University, 80708 Kaohsiung, Taiwan.}
  \email{wong@math.nsysu.edu.tw}

\date{\today}

\begin{abstract}
Let $\bM_n(\IF)$ be the algebra of $n\times n$ matrices over an arbitrary field $\IF$.
We consider linear maps  $\Phi: \bM_n(\IF) \rightarrow \bM_r(\IF)$
preserving  matrices annihilated by a fixed polynomial $f(x) = (x-a_1)\cdots (x-a_m)$ with $m\ge 2$ distinct zeroes $a_1, a_2,   \ldots, a_m  \in \IF$; namely,
$$
f(\Phi(A)) = 0\quad\text{whenever} \quad f(A) = 0.
$$

Suppose that $f(0)=0$, and the zero set
 $Z(f) =\{a_1, \dots, a_m\}$ is not an additive group.
Then $\Phi$ assumes the form
\begin{align}\label{eq:standard}
A \mapsto
S\begin{pmatrix} A \otimes D_1 &&\cr & A^\T \otimes D_2& \cr && 0_s\cr\end{pmatrix}S^{-1}, \tag{$\dagger$}
\end{align}
for some invertible matrix $S\in \bM_r(\IF)$,  invertible diagonal matrices $D_1\in \bM_p(\IF)$ and $D_2\in \bM_q(\IF)$, where
 $s=r-np-nq\geq 0$.
The  diagonal entries $\lambda$ in $D_1$ and
 $D_2$, as well as $0$ in the zero matrix $0_s$, are zero multipliers of $f(x)$ in the sense that
 $\lambda Z(f) \subseteq Z(f)$.

In general, assume that $Z(f) - a_1$ is not an additive group. If 
$\Phi(I_n)$ commutes with $\Phi(A)$ for all $A\in \bM_n(\IF)$, or if $f(x)$ has a unique zero multiplier  $\lambda=1$,
 then $\Phi$ assumes the   form \eqref{eq:standard}.

The above assertions follow from the special case when $f(x)  = x(x-1)=x^2-x$,
for which the  problem  reduces to the study of
linear idempotent preservers.
 It is shown that  a linear map
$\Phi: \bM_n(\IF) \rightarrow \bM_r(\IF)$ sending disjoint rank one idempotents to
disjoint idempotents   always assume the  above form \eqref{eq:standard} with $D_1=I_p$ and $D_2=I_q$, unless
$\bM_n(\IF) = \bM_2(\IZ_2)$.
 \end{abstract}


\maketitle


\section{Introduction}

Linear preserver problems for matrices are well studied and have connections to many other areas; see for example
\cite{Semrl06, LT92,LP,CKLW03, Wong05, W-zpp, Monlar, LCLW18, LTWW20, GLS},
and the references therein.
There has been interest in characterizing a linear map
$\Phi: \bM_n(\IF)\to \bM_r(\IF)$ between matrices over an arbitrary field $\IF$ 
 preserving  matrices annihilated by a fixed polynomial  $f(x)\in \IF[x]$, that is,
$$
f(\Phi(A)) = 0\quad\text{whenever}\quad f(A) = 0;
$$
see, e.g., \cite{Bai,Hou,Howard,Semrl96}. When $f(x) = x(x-1) = x^2-x$, the problem
reduces to the study of idempotent preservers,
that is,
$\Phi(A)^2 = \Phi(A)$ whenever $A^2 = A$. 
Many interesting results 
about the structure of  a linear idempotent preserver  $\Phi$
have been obtained, under the assumptions such as $\Phi$ is surjective, $\Phi$ is unital, and/or
$\IF$ has characteristic $k\neq 2$, etc.;  
see, e.g., \cite{ZTC06,Kuzma}.

In \cite{Hou}, linear maps between complex matrix algebras 
which preserve elements annihilated by a \emph{complex} polynomial
$
f(x) = (x-a_1) \cdots (x-a_m)
$
with $m\geq 2$ distinct zeroes $a_1, \dots, a_m$ are considered.
It is shown that a unital linear map $\Phi$ preserves  matrices annihilated by  $f(x)$
 if and only if $\Phi$ sends idempotents
to idempotents.

In this paper, we consider a 
polynomial  $f(x) = (x-a_1)(x-a_2) \cdots (x-a_m)\in\IF[x]$
with  distinct $m\ge 2$
 zeroes $a_1, a_2,   \ldots, a_m$   in an arbitrary field   $\IF$ of characteristic $k$.
We shall have a complete description of $\Phi$ whenever the zero set   $Z(f) =\{a_1, \dots, a_m\}$ \emph{does not} have
nontrivial algebraic structure; namely, $Z(f)- a_1$ does not form an additive group. This is  automatically the case when $k=0$, and in more general when $m$ is not a power of $k$ (see Remark \ref{rem:poly}).

To prove our results, we first show in Theorem \ref{thm:main}(a) that for $\bM_n(\IF) \ne \bM_2(\IZ_2)$, a linear map $\Phi:\bM_n(\IF)\to \bM_r(\IF)$ assumes the form
\begin{align*}
A \mapsto
S\begin{pmatrix} A \otimes I_p &&\cr & A^\T \otimes I_q &\cr &&0_s \cr\end{pmatrix}S^{-1}
\end{align*}
exactly when $\Phi$ sends disjoint rank one idempotents to disjoint idempotents.
In the case when $\IF$ does not have characteristic 2, it is also equivalent to that
$\Phi$ sending idempotents to idempotents, or that $\Phi$ preserving matrices annihilated by $f(x)=x(x-1)$.
There are counter examples showing that Theorem \ref{thm:main}(a) does not hold if $\bM_n(\IF)=\bM_2(\IF)$;
see Remark \ref{rem:M2Z2-2}.

Let $f(x)\in \IF[x]$ be a polynomial with distinct $m$ zeroes such that the zero set $Z(f)$ is not an additive coset of $\IF$.
We call a scalar $\lambda\in \IF$   a multiplier of the zero set, or simply a  {zero multiplier}, of $f(x)$ if
 $\lambda Z(f)\subseteq Z(f)$.  If $\lambda$ is a nonzero  multiplier of the zero set $Z(f)$ of $f(x)$, then
  $\lambda Z(f)=Z(f)$ and $\lambda^h=1$ for some
positive integer $h$ dividing $m-1$.
Let $\Phi: \bM_n(\IF)\to \bM_r(\IF)$ be a linear map preserving matrices annihilating by $f(x)$.
Using Theorem \ref{thm:main}, we show in  Theorems \ref{thm:sep-poly}, \ref{thm:without1} and \ref{thm:whenPhiunital}  that
$\Phi$ assumes the form
\begin{align}\label{eq:abs-standard}
A \mapsto
S\begin{pmatrix} A \otimes D_1 &&\cr & A^\T \otimes D_2& \cr && 0_s\cr\end{pmatrix}S^{-1},
\end{align}
for some invertible matrix $S\in \bM_r(\IF)$,  invertible diagonal matrices $D_1\in \bM_p(\IF)$ and $D_2\in \bM_q(\IF)$, where
 $s=r-np-nq\geq 0$, whenever
\begin{enumerate}
  \item  $f(0)=0$,
  \item $\Phi(I_n)$ commutes with $\Phi(A)$ for all $A\in \bM_n(\IF)$, or
  \item $f(x)$ has a unique zero multiplier  $\lambda=1$.
\end{enumerate}
The (nonzero)
diagonal  entries $\lambda$ in $D_1$ or $D_2$, as well as $0$ in the zero matrix $0_s$, are zero multipliers of $f(x)$.

Example \ref{eg:Full-AP} shows  that the assumption $Z(f)$  not being an additive coset
 is   necessary.
Moreover, Example \ref{eg:counter-involutions} demonstrates  that $\Phi$ might not always assume the  form \eqref{eq:abs-standard},
in which $\Phi(I_n)$ does not commute with all $\Phi(A)$, and $f(x)=(x-1)(x+1)$ has zero multipliers $\pm1$.

Our paper is organized as follows.
In Section \ref{S:2}, we study linear maps sending
(disjoint) idempotents to (disjoint) idempotents.
 In Section \ref{S:4}, we characterize   linear maps  $\Phi$ between
 matrices satisfying $f(\Phi(A))= 0$ whenever $f(A) = 0$ for a given polynomial
 $f(x)$ with distinct zeroes.
Finally, in Section \ref{S:future} we present some possible extensions of our results.

\section{Linear maps preserving disjoint idempotents}\label{S:2}

In our discussion, we will denote by $\IF$ the underlying field, and
by   $\{E_{11}, E_{12}, \dots, E_{nn}\}$ the standard basis for $\bM_n(\IF)$;
namely, $E_{ij}=e_i e_j^\T$ where $\{e_1,\ldots, e_n\}$ is the standard basis for the (column) vector
space $\IF^n$ over $\IF$.
Here, we write $A^\T=(a_{ji})$ for the transpose of a rectangular matrix $A=(a_{ij})$.
 We also  write $I_k$ and $0_k$ for the $k\times k$ identity matrix and zero matrix, respectively.
Sometimes, we also write $0$ for a rectangular matrix of size determined in context.


\begin{theorem}\label{thm:main}
Let $\Phi: \bM_n(\IF) \rightarrow \bM_r(\IF)$ be a linear map.
Assume that $\Phi$ sends disjoint rank one idempotents to
disjoint idempotents.
\begin{enumerate}[(a)]
  \item Suppose  $\bM_n(\IF)\neq \bM_2(\mathbb{Z}_2)$. Then there are nonnegative integers $p, q$ with $s = n(p+q) \le r$,
and an invertible matrix $S$ in $\bM_r(\IF)$ such that
$\Phi$ assumes the form
\begin{align}\label{oldform-2}
A \mapsto
S\left(
  \begin{array}{ccc}
    A\otimes I_{p} &   &     \\
      &   A^\T \otimes I_{q} &     \\
      &   & 0_{r-s} \\
  \end{array}
\right)S^{-1} \quad\text{for all $A\in \bM_n(\IF)$.}
\end{align}
  \item Suppose  $\bM_n(\IF)= \bM_2(\mathbb{Z}_2)$. Then
  there are  nonnegative integers $k_1, k_2$ with $k_1 +k_2 \leq r$ and an invertible matrix $S$ in $\bM_r(\IF)$ such
    that $\Phi$ assumes the form
\begin{align*}
\begin{pmatrix}
  a & b \\
  c & d \\
\end{pmatrix}
  \mapsto
S\left[\begin{pmatrix}
  aI_{k_1}+bB_{11}+cC_{11} & bB_{12} + cC_{12} \\
  bB_{21} + cC_{21} & dI_{k_2}+bB_{22}+cC_{22}  \\
\end{pmatrix}\oplus 0_{r-k_1-k_2}\right]S^{-1},
\end{align*}
where $B_{ij}, C_{ij}$ are rectangular $k_i\times k_j$ matrices for $i,j=1,2$  satisfying that
\begin{align*}
 \begin{pmatrix}
   B_{11} & B_{12} \\
   B_{21} & B_{22}  \\
 \end{pmatrix}^2   =
 \begin{pmatrix}
   B_{11} & 0  \\
   0 & B_{22}  \\
 \end{pmatrix}
 \quad\text{and}\quad
 \begin{pmatrix}
   C_{11} & C_{12}   \\
   C_{21} & C_{22}   \\
 \end{pmatrix}^2  =
 \begin{pmatrix}
   C_{11} & 0  \\
   0 & C_{22}   \\
 \end{pmatrix}.
\end{align*}
\end{enumerate}
Conversely, if   $\Phi$ assumes the stated form in either case, then $\Phi$ sends disjoint rank one idempotents to
disjoint idempotents.
\end{theorem}

The assertion in Theorem \ref{thm:main}(a) fails when
$\bM_n(\IF)= \bM_2(\mathbb{Z}_2)$, as demonstrated below.

\begin{example}\label{rem:M2Z2-2}\label{rem:M2Z2-1}
Note that all
disjoint rank one idempotent pairs  in $\bM_2(\mathbb{Z}_2)$ are:
$$
\begin{pmatrix}
  1 & 0\\
  0 & 0 \\
\end{pmatrix}\ \text{and}\ \begin{pmatrix}
  0 & 0\\
  0 & 1 \\
\end{pmatrix}, \quad
\begin{pmatrix}
  0 & 1\\
  0 & 1 \\
\end{pmatrix}\ \text{and}\ \begin{pmatrix}
  1 & 1\\
  0 & 0 \\
\end{pmatrix}, \quad 
\begin{pmatrix}
  1 & 0\\
  1 & 0 \\
\end{pmatrix}\ \text{and}\ \begin{pmatrix}
  0 & 0\\
  1 & 1 \\
\end{pmatrix}.
$$
One can verify directly that the linear maps from $\bM_2(\mathbb{Z}_2)$ into $\bM_2(\mathbb{Z}_2)$, defined respectively by
$$
\begin{pmatrix}
  a & b\\
  c & d \\
\end{pmatrix}
\mapsto
\begin{pmatrix}
  a & 0\\
  0 & d \\
\end{pmatrix} \quad\text{or}\quad
 \begin{pmatrix}
  a & b \\
  c & d \\
\end{pmatrix}
  \mapsto
\begin{pmatrix}
  a+b & 0 \\
  0 & 0 \\
 \end{pmatrix},
$$
send disjoint rank one idempotents to disjoint  idempotents.
The same is true for the linear map from $\bM_2(\IZ_2)$ into $\bM_3(\IZ_2)$ defined by
$$
\begin{pmatrix}
  a & b \\
  c & d \\
\end{pmatrix}
  \mapsto
\begin{pmatrix}
  a & b & b + c \\
  0 & d & 0 \\
  0 & 0 & d \\
\end{pmatrix}.
$$
 However, neither of these three maps assumes the form \eqref{oldform-2}.
\end{example}

\begin{proof}[Proof of Theorem \ref{thm:main}(b)]
    Let $\Phi: \bM_2(\IZ_2)\to \bM_r(\IZ_2)$ be
  a linear map sending disjoint rank one idempotents to disjoint idempotents.
  Since $E_{11}, E_{22}$ are disjoint rank one idempotents, $\Phi(I_2)=\Phi(E_{11})+\Phi(E_{22})$ is a sum
  of disjoint idempotents, and thus an idempotent.  After a similarity transformation, we can assume that
  $$
  \Phi(E_{11})=  I_{k_1}\oplus 0_{k_2} \oplus 0_{r-k_1 - k_2} \quad\text{and}\quad
  \Phi(E_{22})=  0_{k_1}\oplus I_{k_2} \oplus 0_{r-k_1 - k_2},
  $$
  for some nonnegative integers $k_1, k_2$ with $r-k_1 - k_2\geq 0$.
    Write
    $$
    \Phi(E_{12})= \begin{pmatrix}
   B_{11} & B_{12} & B_{13} \\
   B_{21} & B_{22} & B_{23} \\
   B_{31} & B_{32} & B_{33} \\
 \end{pmatrix}
 $$
 as a block matrix accordingly.
  Since $E_{11} + E_{12}$ and $E_{12} + E_{22}$ are disjoint  rank one idempotents, we have
  $\Phi(E_{11} + E_{12})\Phi(E_{12}+E_{22})= \Phi(E_{12}+E_{22})\Phi(E_{11} + E_{12})=0$.  This gives
\begin{align*}
\Phi(E_{12})^2 &= \Phi(E_{11})\Phi(E_{12})+ \Phi(E_{12})\Phi(E_{22}) = \Phi(E_{12})\Phi(E_{11})+ \Phi(E_{22})\Phi(E_{12}) \\
                &=
                \begin{pmatrix}
   B_{11} & 0 & B_{13} \\
   0 & B_{22} & 0 \\
   0 & B_{32} & 0 \\
 \end{pmatrix} =
   \begin{pmatrix}
   B_{11} & 0 & 0 \\
   0 & B_{22} & B_{23} \\
   B_{31} & 0 & 0 \\
 \end{pmatrix}.
\end{align*}
On the other hand,
  $\Phi(E_{11} + E_{12})^2 = \Phi(E_{11} + E_{12})$.  This gives
  $$
  \Phi(E_{12})^2 = \Phi(E_{11})\Phi(E_{12}) + \Phi(E_{12})\Phi(E_{11}) + \Phi(E_{12}) =
  \begin{pmatrix}
   B_{11} & 0 & 0 \\
   0 & B_{22} & B_{23} \\
   0 & B_{32} & B_{33} \\
 \end{pmatrix}.
$$
In particular, $B_{13}, B_{23}, B_{31}, B_{32}$ and $B_{33}$ are all zero rectangular matrices.
Therefore,
\begin{align}\label{eq:B}
\Phi(E_{12})  =
 \begin{pmatrix}
   B_{11} & B_{12} & 0 \\
   B_{21} & B_{22} & 0 \\
   0 & 0 & 0 \\
 \end{pmatrix} \quad\text{such that}\quad
\Phi(E_{12})^2  =
 \begin{pmatrix}
   B_{11} & 0 & 0 \\
   0 & B_{22} & 0 \\
   0 & 0 & 0 \\
 \end{pmatrix}.
\end{align}
Similarly, we have
\begin{align}\label{eq:C}
\Phi(E_{21})  =
 \begin{pmatrix}
   C_{11} & C_{12} & 0 \\
   C_{21} & C_{22} & 0 \\
   0 & 0 & 0 \\
 \end{pmatrix} \quad\text{such that}\quad
\Phi(E_{21})^2  =
 \begin{pmatrix}
   C_{11} & 0 & 0 \\
   0 & C_{22} & 0 \\
   0 & 0 & 0 \\
 \end{pmatrix}.
\end{align}
In summary, after a similarity transformation, we can derive that
  $\Phi$ must assume the expected form
\begin{align*}\label{eq:M2Z2}
\begin{pmatrix}
  a & b \\
  c & d \\
\end{pmatrix}
  \mapsto
\begin{pmatrix}
  aI_{k_1}+bB_{11}+cC_{11} & bB_{12} + cC_{12} \\
  bB_{21} + cC_{21} & dI_{k_2}+bB_{22}+cC_{22}  \\
\end{pmatrix}\oplus 0_{r-k_1-k_2},
\end{align*}
with rectangular matrices $B_{ij}, C_{ij}$ satisfying \eqref{eq:B} and \eqref{eq:C}.
By direct computations, we see
 that any map $\Phi$ assuming the above form sends disjoint rank one idempotents to disjoint idempotents.
\end{proof}

In the following we establish some lemmas needed for the proof of Theorem \ref{thm:main}(a).
We  start with the following simple observations.

\begin{lemma} \label{P1P2}\label{nil}
Let $\IF$ be a field and $n$ be a positive integer with $n \ge 2$.
\begin{itemize}
\item[{\rm (a)}]
The linear space $\bM_n(\IF)$ has the following basis consisting
of rank one idempotents
$$\{E_{jj}: 1 \le j \le n\} \cup \{E_{ii} + E_{ij}: 1 \le i \le n, i\ne j\}.$$
For $i \ne j$,  $E_{ii} + E_{ij}$ and $E_{jj}-E_{ij}$
are disjoint rank one idempotents.

\item[{\rm (b)}]
Suppose $\IF \ne \IZ_2$. There is $a\ne 0$ with $a^2 + 1 \neq 0$ such that
\begin{align*}
P_1 &=(1+a^2)^{-1}(a^2 E_{ii} + a(E_{ij}+E_{ji}) + E_{jj})\quad\text{and}\\
P_2 &= (1+a^2)^{-1}( E_{ii} - a(E_{ij}+E_{ji}) + a^2E_{jj})
\end{align*}
are disjoint rank one idempotents whenever $i\ne j$.

\item[{\rm (c)}] ({See, e.g., \cite[Theorem A.0.4] {ZTC06}})
Every $A \in \bM_r(\mathbb{F})$ is similar to a direct sum
 $R\oplus N$  of an invertible matrix $R$ in $\bM_{r-s}(\IF)$ and a nilpotent matrix  $N$ in $\bM_s(\IF)$ with a nonnegative
 integer $s\leq r$, such that
$N$ is a direct sum of upper triangular Jordan blocks
for the eigenvalue zero of $A$.
Here, $R$ or $N$ can be vacuous.
In particular, if $A$ is an idempotent, then $R = I_{r-s}$ and $N = 0_s$.
\end{itemize}
\end{lemma}

\begin{lemma} \label{Ejj}
Suppose $\Phi: \bM_n(\IF) \rightarrow \bM_r(\IF)$ is linear and sends disjoint rank one
idempotents to disjoint idempotents. Then there is an invertible $S \in \bM_r(\IF)$ and
nonnegative integers $k_1, \dots, k_n$ with $s= r-(k_1 + \cdots + k_n)\geq 0$
such that
\begin{equation}\label{Ejj-form}
S^{-1}\Phi(E_{jj})S = 0_{k_1} \oplus \cdots \oplus 0_{k_{j-1}}
\oplus I_{k_j} \oplus 0_{k_{j+1}} \oplus \cdots \oplus 0_{k_n} \oplus 0_{s}.
\end{equation}
Moreover, if  $\bM_n(\IF)\neq \bM_2(\IZ_2)$ we have
\begin{equation}\label{Eij-form}
S^{-1}\Phi(E_{ij})S = \left\{ \begin{array}{ll}
((E_{ij}\otimes B_{ij}) + (E_{ji}\otimes B_{ji})) \oplus 0_{s}     & \mbox{ if $i<j$,} \\
 ((E_{ij}\otimes C_{ij}) + (E_{ji}\otimes C_{ji})) \oplus 0_{s}     & \mbox{ if $i>j$.}
\end{array}
\right. \end{equation}
with $B_{ij} B_{ji} = C_{ij} C_{ji} = 0_{k_i}$  and $B_{ji} B_{ij} = C_{ji} C_{ij} = 0_{k_j}$
for distinct $i,j = 1, \dots, n$.
\end{lemma}

\it Proof. \rm
Note that $\Phi(E_{11}), \dots, \Phi(E_{nn})$
are idempotents and $\Phi(E_{ii})\Phi(E_{jj})=0$ if $i\neq j$.
Let $k_i$ be the rank of the idempotent $\Phi(E_{ii})$ for $i=1,\ldots, n$, and
thus the idempotent $\Phi(I_n) = \sum_{i=1}^n \Phi(E_{ii})$ has rank $k_1 + \cdots + k_n$.
  We can find a basis of the vector space $\IF^r$ which
is a direct sum of the range spaces of $\Phi(E_{11}), \ldots, \Phi(E_{nn})$ and a subspace of dimension
$s = r-( k_1 + \cdots + k_n)$
 complement to their sum.
Using these basic vectors as columns, we can find an invertible $S\in \bM_r(\IF)$ such that
$S^{-1}\Phi(E_{jj})S$ has the asserted form  (\ref{Ejj-form}).
Assuming that  (\ref{Ejj-form})
holds with $S = I_r$.

\smallskip

\noindent\textbf{Case 1.} Assume that $\IF$ does not have characteristic 2.

We consider $\Phi(E_{ij})$ with $i\ne j$ in the following.
Let
$$
\Phi(E_{12}) =Z = (Z'_{ij})_{1 \le i, j \le n+1}\in \bM_{r}(\IF)
$$
for rectangular $k_i\times k_j$ matrix $Z'_{ij}$   with $i, j = 1, \dots, n+1$ and
$k_{n+1} = s$.
Let
$$
X_1 = E_{11} + E_{12},\quad X_2 = E_{22}-E_{12},\quad
Y_1 = E_{11} -  E_{12}\quad\text{and}\quad Y_2 = E_{22} + E_{12}.
$$
Then $\{X_1, X_2\}, \{Y_1, Y_2\}$
are disjoint rank one idempotent pairs. Hence,
\begin{align}\label{x1x2-form}
0 &= \Phi(X_1)\Phi(X_2) = (\Phi(E_{11}) + Z)(\Phi(E_{22})-Z) =
Z\Phi(E_{22}) -  \Phi(E_{11})Z - Z^2, \notag \\
0 &= \Phi(Y_1)\Phi(Y_2) = (\Phi(E_{11}) - Z)(\Phi(E_{22})+Z) =
- Z\Phi(E_{22}) + \Phi(E_{11})Z - Z^2.
\end{align}

By \eqref{x1x2-form}, $Z^2=Z\Phi(E_{22}) - \Phi(E_{11})Z = 0$.  This
implies that $Z'_{1j}=0$ for all $j\ne 2$
and $Z'_{i2} = 0$ for all $i \ne 1$.
Similarly, by the fact that
$0 = \Phi(X_2)\Phi(X_1) = \Phi(Y_2)\Phi(Y_1)$, we see that
$Z'_{2j} = 0$ for all $j \ne 1$ and $Z'_{i1} = 0$  for $i \ne  2$.

On the other hand, since $\Phi(E_{11} + E_{12})^2= \Phi(E_{11} + E_{12})$ and $Z^2=0$, we have
$$
Z = (I_{k_1}\oplus 0_{r-k_1})Z + Z(I_{k_1}\oplus 0_{r-k_1}) = \begin{pmatrix}
                            0 & Z'_{12} \\
                            Z'_{21} & 0 \\
                          \end{pmatrix} \oplus 0_{r-k_1-k_2}
                          = (E_{12}\otimes Z'_{12}) + (E_{21}\otimes Z'_{21}) \oplus 0_{r-k_1-k_2}.
$$
Because $Z^2=0$, we have $Z'_{12}Z'_{21} = 0_{k_1}$ and $Z'_{21}Z'_{12}=0_{k_2}$.
Now we might set $B_{12}=Z'_{12}$ and $B_{21}=Z'_{21}$, and write
$$
\Phi(E_{12})= \big((E_{12}\otimes B_{12}) + (E_{21}\otimes B_{21})\big) \oplus 0_{r-k_1-k_2}.
$$

Arguing in a similar way for $\Phi(E_{21})$, we can write
$$
\Phi(E_{21})= \big((E_{12}\otimes C_{12}) + (E_{21}\otimes C_{21})\big) \oplus 0_{r-k_1-k_2}
$$
for some rectangular matrices $C_{12}, C_{21}$ such that $C_{12}C_{21}= 0_{k_1}$ and $C_{21}C_{12}=0_{k_2}$.
In general, we will have similar relations for the matrix block entries of other $\Phi(E_{ij})$ as stated
in \eqref{Eij-form}.

\smallskip

\noindent\textbf{Case 2.} Assume that $\IF$  has characteristic 2, $n=2$, but $\IF\neq \IZ_2$.

Let
$$
\Phi(E_{12}) =Z = (Z'_{ij})_{1 \le i, j \le n+1}\in \bM_{r}(\IF)
$$
as in Case 1.  Let $a\in \IF$ with $a^2\neq 1$,
and
$$
X_1 = E_{11} + aE_{12},\quad X_2 = E_{22}+aE_{12},\quad
Y_1 = E_{11} + a^{-1}E_{12}\quad\text{and}\quad Y_2 = E_{22} + a^{-1}E_{12}.
$$
Then $\{X_1, X_2\}, \{Y_1, Y_2\}$  are disjoint rank one idempotent pairs.
Arguing  as in Case 1, we see that
$$
Z^2=a^{-1}(Z\Phi(E_{22}) + \Phi(E_{11})Z) = a(Z\Phi(E_{22}) + \Phi(E_{11})Z)=0,
$$
and we establish \eqref{Eij-form} in the same way.

\smallskip

\noindent\textbf{Case 3.} Assume that $\IF$ has characteristic 2 and $n\geq 3$.

Using the arguments in the proof of Theorem  \ref{thm:main}(b), we see that
\begin{align*}
\Phi(E_{12})  &=
 \begin{pmatrix}
   B_{11} & B_{12} \\
   B_{21} & B_{22} \\
  \end{pmatrix}\oplus 0_{r-k_1-k_2} \quad\text{and}\quad
\Phi(E_{12})^2  =
 \begin{pmatrix}
   B_{11} & 0 \\
   0 & B_{22} \\
   \end{pmatrix}\oplus 0_{r-k_1-k_2}, \\
\Phi(E_{21})  &=
 \begin{pmatrix}
   C_{11} & C_{12} \\
   C_{21} & C_{22} \\
   \end{pmatrix}\oplus 0_{r-k_1-k_2} \quad\text{and}\quad
\Phi(E_{21})^2  =
 \begin{pmatrix}
   C_{11} & 0  \\
   0 & C_{22}  \\
 \end{pmatrix}\oplus 0_{r-k_1-k_2}.
\end{align*}
Similarly, we can show that all $(p, q)$ blocks of $\Phi(E_{ij})$ are 0 except for $(p,q) = (i,i), (i,j), (j,i), (j,j)$,
and $\Phi(E_{ij})^2$ equals the diagonal of $\Phi(E_{ij})$.

Considering the disjoint rank one idempotent pairs
\begin{gather*}
\{E_{11}+E_{12}+E_{1n}, E_{12}+E_{22}\}, \quad  \{E_{11}+E_{12}, E_{12}+E_{22}\}, \\
\{E_{11}+E_{12}, E_{12}+E_{22}+E_{n2}\} \quad \text{and}\quad \{E_{11}+E_{12}, E_{12}+E_{22}\},
\end{gather*}
 we see that
$$
\Phi(E_{1n})\Phi(E_{12})=0 \quad\mbox{and}\quad \Phi(E_{12})\Phi(E_{n2})=0.
$$
Since $\{E_{11}+E_{1n}, E_{12}+E_{22}+E_{n2}\}$ is a disjoint rank one idempotent pair, we have
$$
0=\Phi(E_{11}+E_{1n})\Phi(E_{12}+E_{22}+E_{n2})=\Phi(E_{11})\Phi(E_{12})+\Phi(E_{1n})\Phi(E_{n2}).
$$
Because the $(1, 1)$ block of $\Phi(E_{11})\Phi(E_{12})+\Phi(E_{1n})\Phi(E_{n2})$ is $B_{11}$, we have $B_{11}=0$.
Since $\{E_{11}+E_{12}+E_{1n}, E_{22}+E_{n2}\}$ is also a disjoint rank one idempotent pair, we have
$$
0=\Phi(E_{11}+E_{12}+E_{1n})\Phi(E_{22}+E_{n2})=\Phi(E_{12})\Phi(E_{22})+\Phi(E_{1n})\Phi(E_{n2}).
$$
Because  the $(2, 2)$ block of $\Phi(E_{12})\Phi(E_{22})+\Phi(E_{1n})\Phi(E_{n2})$ is $B_{22}$, we have $B_{22}=0$.
Hence,
\begin{align*}
\Phi(E_{12}) &= \begin{pmatrix} 0_{k_1} & B_{12} \cr B_{21} & 0_{k_2} \cr\end{pmatrix}
\oplus 0_{r-k_1-k_2} \quad \hbox{ with }  \
B_{12} B_{21} = 0_{k_1} \ \hbox{ and } \ B_{21} B_{12} = 0_{k_2}.
\end{align*}

In a similar manner, we can obtain the asserted form (\ref{Eij-form}) if we replace $(1,2, n)$ by $(i, j, k)$ for some $k\neq i,j$.
\qed

\medskip\noindent
{\bf  Proof of Theorem \ref{thm:main}(a).} \rm
 It is clear that if $\Phi$ assumes the form \eqref{oldform-2} then $\Phi$ sends disjoint
 rank one idempotents to disjoint idempotents.
 Assume now that the linear map $\Phi: \bM_n(\IF)\to \bM_r(\IF)$ sends disjoint
 rank one idempotents to disjoint idempotents and $\bM_n(\IF)\neq \bM_2(\IZ_2)$.
Assume also that $\Phi$ is nonzero to avoid trivial consideration.

By Lemma \ref{Ejj}, replacing $\Phi$ by the map
$A \mapsto S^{-1}\Phi(A)S$, we can assume that  (\ref{Ejj-form})  and (\ref{Eij-form})
hold with $S = I_r$.
Denote by
$$
Y_{ij} = B_{ij} + C_{ij},
$$
which is a  $k_i\times k_j$ rectangular matrix for  $i\neq j$  in between $1$ and $n$.

\smallskip

\noindent\textbf{Step 1.} Suppose  that $\IF\neq \mathbb{Z}_2$.
Let $s'=r-k_1-k_2$, and write
$$
\Phi(E_{12}+E_{21}) = \begin{pmatrix} 0_{k_1} & Y_{12} \cr Y_{21} & 0_{k_2}\cr
\end{pmatrix} \oplus 0_{s'}.
$$
By Lemma \ref{P1P2}, there is $a\ne 0$ in $\IF$
such that $a^2+ 1 \ne 0$, and
$$P_1 = (1+a^2)^{-1} (E_{11} + a^2 E_{22} + a(E_{12}+E_{21}))$$
and
$$P_2 = (1+a^2)^{-1}(a^2 E_{11} + E_{22} - a (E_{12}+E_{21}))$$
are
disjoint rank one idempotents in $\bM_n(\IF)$.
Thus,
\begin{align*}
0_r & =  (1+a^2)^{2}\Phi(P_1)\Phi(P_2)\\
& =
\Phi(E_{11} + a^2 E_{22} + a(E_{12}+E_{21}))
\Phi(a^2E_{11}+ E_{22} - a(E_{12}+E_{21})) \\
& =
\left[\begin{pmatrix}  I_{k_1} & 0  \cr 0 & a^2I_{k_2}  \cr \end{pmatrix} \oplus 0_{s'} + a Y\right]
\left[\begin{pmatrix} a^2I_{k_1} & 0  \cr 0 & I_{k_2}  \cr \end{pmatrix} \oplus 0_{s'} -a Y\right]\\
&=  a^2 \left[\begin{pmatrix}  I_{k_1} & 0  \cr 0 & I_{k_2}  \cr \end{pmatrix}
\oplus 0_{s'}\right ] - a^2 Y^2.
\end{align*}
Therefore, $Y_{12}Y_{21} = I_{k_1}$ and $Y_{21}Y_{12} = I_{k_2}$.
Hence, $k_1 = k_2$ and $Y_{21} = Y_{12}^{-1}$.  Set $k=k_1$.

We apply the arguments for
$\{E_{11}, E_{22}, E_{12}, E_{21}\}$ to  $\{E_{ii}, E_{jj}, E_{ij}, E_{ji}\}$
with  $1 \le i < j \le n$ and $j\geq 3$.
We can conclude that $k_i=k_j = k$ and
\begin{equation}\label{Eij-1}
\Phi(E_{ij}+E_{ji}) = (E_{ij}\otimes Y_{ij} + E_{ji}\otimes Y_{ji}) \oplus 0_{r-nk}\quad
\hbox{with some } Y_{ji}=Y_{ij}^{-1}\in \bM_k(\IF).
\end{equation}

\smallskip

\noindent\textbf{Step 2.}
Suppose   that $\IF=\mathbb{Z}_2$ and $n\geq 3$.
Consider the rank one idempotents in $\bM_n(\mathbb{Z}_2)$:
\begin{gather*}
E= \begin{pmatrix}
     1 & 1 & 1 \\
     1 & 1 & 1 \\
     1 & 1 & 1 \\
   \end{pmatrix} \oplus 0_{n-3}, \quad
F_1 =  \begin{pmatrix}
     1 & 0 & 1 \\
     1 & 0 & 1 \\
     0 & 0 & 0 \\
   \end{pmatrix} \oplus 0_{n-3} \quad\text{and}\quad
F_2 =  \begin{pmatrix}
     0 & 0 & 0 \\
     1 & 1 & 0 \\
     1 & 1 & 0 \\
   \end{pmatrix} \oplus 0_{n-3}.
\end{gather*}
By Lemma \ref{Ejj} we can write
\begin{gather*}
\Phi(E)= \begin{pmatrix}
     I_{k_1} & Y_{12} & Y_{13} \\
     Y_{21} & I_{k_2} & Y_{23} \\
     Y_{31} & Y_{32} & I_{k_3} \\
   \end{pmatrix}\oplus 0_{r-k_1-k_2-k_3}, \\
\Phi(F_1) =  \begin{pmatrix}
     I_{k_1} & C_{12} & B_{13} \\
     C_{21} & 0 & B_{23} \\
     B_{31} & B_{32} & 0 \\
   \end{pmatrix}\oplus 0_{r-k_1-k_2-k_3} \quad\text{and} \quad
\Phi(F_2) =  \begin{pmatrix}
     0 & C_{12} & C_{13} \\
     C_{21} & I_{k_2} & C_{23} \\
     C_{31} & C_{32} & 0 \\
   \end{pmatrix}\oplus 0_{r-k_1-k_2-k_3}.
\end{gather*}
Since $E$ is disjoint from both $F_1, F_2$, looking at the $(1,1)$ and $(2,2)$ entries of the zero products
$\Phi(E)\Phi(F_1)$, $\Phi(F_1)\Phi(E)$, $\Phi(F_2)\Phi(E)$  and $\Phi(E)\Phi(F_2)$, respectively, we have
\begin{align*}
Y_{12}C_{21}+ Y_{13}B_{31} & = C_{12}Y_{21}+ B_{13}Y_{31}= I_{k_1},\\
Y_{12}C_{21}+Y_{13}C_{31} & = C_{12}Y_{21}+C_{13}Y_{31} =0_{k_1},\\
C_{21}Y_{12} + C_{23}Y_{32} &  = Y_{21}C_{12} + Y_{23}C_{32}= I_{k_2}, \\
Y_{21}C_{12}+Y_{23}B_{32} & = C_{21}Y_{12}+B_{23}Y_{32} = 0_{k_2}.
\end{align*}
Since $B_{ij}B_{ji}=C_{ij}C_{ji}=0_{k_i}$ for all $i\neq j$, we have
\begin{align}
B_{12}C_{21} + C_{13}B_{31} &= B_{13}C_{31} +C_{12}B_{21} =  I_{k_1}, \label{eq:b12}\\
B_{12}C_{21} + B_{13}C_{31} &= C_{12}B_{21} + C_{13}B_{31}=0_{k_1}, \label{eq:b21}\\
C_{21}B_{12} + C_{23}B_{32} &= B_{21}C_{12} +B_{23}C_{32}= I_{k_2}, \label{eq:c12}\\
B_{21}C_{12} + C_{23}B_{32} &= C_{21}B_{12} + B_{23}C_{32}=0_{k_2} \label{eq:c21}.
\end{align}
From \eqref{eq:b21} and \eqref{eq:c21}, we have
$$
C_{13}B_{31} = C_{12}B_{21} \quad\text{and}\quad B_{23}C_{32} = C_{21}B_{12}.
$$
Therefore, \eqref{eq:b12} and \eqref{eq:c12} provide that
\begin{align*}
Y_{12}Y_{21} &= (B_{12}+C_{12})(B_{21}+C_{21})=B_{12}C_{21} + C_{12}B_{21} = B_{12}C_{21} + C_{13}B_{31} =I_{k_1}, \\
Y_{21}Y_{12} &= (B_{21}+C_{21})(B_{12}+C_{12})=B_{21}C_{12} + C_{21}B_{12} = B_{21}C_{12} + B_{23}C_{32} =I_{k_2}.
\end{align*}
Consequently,
$ k_1 = k_2$ and $Y_{12} = Y_{21}^{-1}$.
Applying similar arguments to other indices, we can conclude that
$k_1=k_2 =\cdots = k_n :=k$, and $Y_{ij}=Y_{ji}^{-1}$ for all $i\neq j$.
In particular, \eqref{Eij-1} also holds in this case.

\smallskip

\noindent\textbf{Step 3.}
Assuming \eqref{Eij-1} from now on.
 We may replace $\Phi$ with the map
$$
X \mapsto
(I_k \oplus Y_{12} \oplus Y_{13} \oplus  \cdots \oplus Y_{1n}\oplus I_{r-nk})\Phi(X)
(I_k \oplus Y_{12}^{-1} \oplus Y_{13}^{-1} \oplus  \cdots \oplus Y_{1n}^{-1}\oplus I_{r-nk}),
$$
so that
\begin{equation}\label{Eij-2}
\Phi(E_{1j}+E_{j1}) = ((E_{1j}+E_{j1})\otimes I_k)\oplus 0_{r-nk} \quad\text{for $j=2,\ldots, n$.}
\end{equation}
Since $\Phi(E_{12} + E_{21}) = ((E_{12} + E_{21})\otimes I_k)\oplus 0_{r-nk}$, we see that
$B_{12}+C_{12}=I_k$ and $B_{21}+C_{21}=I_k$.
By Lemma \ref{nil}(c),
there are invertible matrices
$U\in \bM_{k}(\IF)$, $R\in \bM_{p}(\IF)$ and a nilpotent matrix $N\in \bM_{q}(\IF)$
such that $p+q=k$ and
$$
\hat B_{12} := UB_{12}U^{-1}=\begin{pmatrix} R & 0 \cr
0 & N \cr\end{pmatrix}.
$$
If $\hat B_{21} := UB_{21}U^{-1}$,
then  $\hat B_{12}\hat B_{21} = \hat B_{21}\hat B_{12} = 0$.
Thus, $\hat B_{21} = 0_{p} \oplus T$ with $T \in \bM_q(\IF)$
satisfying $TN = NT = 0_q$.
Since $B_{12} + C_{12} = B_{21} + C_{21} = I_k$,
we have
$$\hat C_{12} := UC_{12}U^{-1} = (I_p - R) \oplus (I_q - N) \quad \hbox{
and } \quad \hat C_{21} := UC_{21}U^{-1} = I_p \oplus (I_q - T).$$
Since
$(I_q-N) \in \bM_q(\IF)$ is invertible, and
$$ 0_k = C_{12}C_{21} = \hat C_{12} \hat C_{21} =
(I_p-R)I_p \oplus (I_q-N)(I_q-T),$$
we see that $R = I_p$ and $I_q = T$, and thus $N = 0_q$.

Replace $\Phi$ by the map $A\mapsto (I_n\otimes U)\Phi(A)(I_n\otimes U^{-1})$.
Then
$$
\Phi(E_{12})=\begin{pmatrix} 0_{p} & 0 & I_{p} & 0\cr
0 & 0_{q} & 0 & 0_{q} \cr
0_{p} & 0 & 0_{p} & 0 \cr
0 & I_{q} & 0 & 0_{q} \cr\end{pmatrix} \oplus 0_{r-2k}
\hspace{5mm}\mbox{and}\hspace{5mm}
\Phi(E_{21})=\begin{pmatrix} 0_{p} & 0 & 0_{p} & 0\cr
0 & 0_{q} & 0 & I_{q} \cr
I_{p} & 0 & 0_{p} & 0 \cr
0 & 0_{q} & 0 & 0_{q} \cr\end{pmatrix} \oplus 0_{r-2k}.
$$
Suppose $n \ge 3$.
It follows from \eqref{Eij-form} that
\begin{equation}\label{Eij-3}
\Phi(E_{ij})= \Phi(E_{ii})\Phi(E_{ij}) + \Phi(E_{ij})\Phi(E_{ii}) \quad \hbox{ for any  } i \ne j.
\end{equation}
For any $j = 3, \dots, n$,
$\Phi(E_{11} + E_{12} + E_{1j})$ and $\Phi(E_{11} + E_{21} + E_{j1})$ are both idempotents.  Thus, by (\ref{Eij-3}),
$$0 = [\Phi(E_{11} + E_{12} + E_{1j})]^2 -\Phi(E_{11} + E_{12} + E_{1j}) =
\Phi(E_{12})\Phi(E_{1j})+\Phi(E_{1j})\Phi(E_{12}),$$
$$0 = [\Phi(E_{11} + E_{21} + E_{j1})]^2 - \Phi(E_{11} + E_{21} + E_{j1}) =
\Phi(E_{21})\Phi(E_{j1})+\Phi(E_{j1})\Phi(E_{21}).
$$
By (\ref{Eij-2}), $B_{1j}+C_{1j}=B_{j1}+C_{j1}=I_k$.
The above equations lead to the conclusion that
$$
B_{1j}=C_{j1}=I_{p}\oplus 0_{q} \hspace{5mm}\mbox{and}\hspace{5mm}
B_{j1}=C_{1j}=0_{p}\oplus I_{q}.
$$

We are going to show that
\begin{equation}\label{Eij-4}
\Phi(E_{ij}+E_{ji}) = (E_{ij}+E_{ji})\otimes I_k \quad\text{whenever\ $2\le i < j\le n$.}
\end{equation}
To see this,  let $X = uv^\T/(2+a)$ with
$u = ae_1 + e_i + e_j$ and $v = e_1 + e_i + e_j$, where $a = 2$ if
$\IF$ has characteristic 3, and $a = 1$ otherwise.
Then $X$ is a rank one idempotent so that,
up to a permutation, $\Phi(X)$ is a direct sum of the zero matrix $0_{r-3k}$ and
$$ Z = (a+2)^{-1}\begin{pmatrix}
aI_{k} & a I_{p}\oplus I_{q} & a I_{p}\oplus I_{q} \cr
I_{p}\oplus aI_{q} & I_k & Y_{ij}\cr
I_{p}\oplus aI_{q} & Y_{ij}^{-1} & I_k\cr\end{pmatrix}.$$
Hence,
$Z^2 =  Z$. Considering the $(2,3)$ block of $Z^2$ and $Z$,
we see that $aI_k + 2Y_{ij} = (a+2) Y_{ij}$. Hence, $Y_{ij} = I_k$,
and thus $Y_{ji} = Y_{ij}^{-1} = I_k$.

Now, for $2 \le i < j \le n$,
$\Phi(E_{i1} + E_{ii} + E_{ij})$ and $\Phi(E_{1i} + E_{ii} + E_{ji})$ are both idempotents.  We can use the previous arguments to conclude that
$$
\Phi(E_{i1})\Phi(E_{ij})+\Phi(E_{ij})\Phi(E_{i1})=0
\hspace{5mm}\mbox{and}\hspace{5mm}
\Phi(E_{1i})\Phi(E_{ji})+\Phi(E_{ji})\Phi(E_{1i})=0.
$$
By a direct calculation and (\ref{Eij-4}), we have
$$
B_{ij}=C_{ji}=I_{p}\oplus 0_{q} \hspace{5mm}\mbox{and}\hspace{5mm} B_{ji}=C_{ij}=0_{p}\oplus I_{q}.
$$
After a permutation similarity, we have
$$
 \Phi(A) = (A \otimes I_{p}) \oplus (A^\T \otimes I_{q})\oplus 0_{r-nk}
$$
for $A = E_{ij}$ for all $i,j$, and the result follows by
linearity.
\qed

Recall that a linear map $\Phi: \bM_n(\IF) \rightarrow \bM_r(\IF)$
is a Jordan homomorphism if $\Phi(AB+BA) = \Phi(A)\Phi(B) + \Phi(B)\Phi(A)$
for all $A, B \in \bM_n(\IF)$.
It is known and easy to check that
if the characteristic of $\IF$ is not 2, then the above
condition is equivalent to $\Phi(A^2) = \Phi(A)^2$ for all $A \in \bM_n(\IF)$; in this case,
$\Phi $ sends idempotents to idempotents if and only if $\Phi$ assumes the form \eqref{oldform-2} (\cite[Corollary 6.1.8]{ZTC06}).
When the characteristic of $\IF$ is not 2,
the sum of two idempotents $A$ and $B$ is an idempotent
if and only if $A$ and $B$ are disjoint.
To see this, let
$A, B \in \bM_n(\IF)$ be idempotents.  Then
$(A+B)^2 = A+B$ if and only if
$AB+BA =0_n$.  If $AB+BA = 0_n$ then $A^2B + ABA = ABA + BA^2 =  0_n$.
Since $A$ is an idempotent,  $AB = BA = 0_n$.
The other implication is trivial.

\begin{corollary}\label{cor:J-homo}
Let  $\bM_n(\IF) \ne \bM_2(\mathbb{Z}_2)$. Suppose  $\Phi: \bM_n(\IF) \rightarrow \bM_r(\IF)$ is a  linear map sending
rank one disjoint idempotents to
disjoint idempotents, or equivalently, $\Phi$ assumes the form \eqref{oldform-2}.
Then the following conditions hold.
\begin{itemize}
\item[{\rm (a)}] $\Phi$ is a Jordan homomorphism.
\item[{\rm (b)}] $\Phi(A^2) = \Phi(A)^2$ for all $A \in\bM_n(\IF)$.
\item[{\rm (c)}] $\Phi$ sends idempotents to idempotents.
\item[{\rm (d)}] $\Phi$ sends idempotents of rank at most 2 to
idempotents.
\end{itemize}
In case when the characteristic of $\IF$ is not 2, conditions {\rm (a) -- (d)}
are equivalent to each other, as well as  the condition that $\Phi$ assumes the form \eqref{oldform-2}.
\end{corollary}

\it Proof. \rm  If $\Phi$ assumes the form \eqref{oldform-2},
then it is clear that conditions {\rm (a) -- (d)} hold. Suppose the characteristic of $\IF$ is not 2.
The implications
 (a) $\Rightarrow$ (b)
$\Rightarrow$ (c) $\Rightarrow$ (d) are clear.
 Assume (d) holds. Then
$\Phi$ sends disjoint rank one idempotents to disjoint idempotents
by the discussion before the corollary. Therefore, $\Phi$ assumes the form
(\ref{oldform-2}) by Theorem \ref{thm:main}.
\qed

\begin{example}  \label{eg:ch2-id-not}
When $\IF$ has characteristic $2$, one cannot conclude that $\Phi$ assumes the form \eqref{oldform-2}
by using any of the conditions {\rm (a) -- (d)} in Corollary \ref{cor:J-homo}.
For example, consider the linear map $\Phi: \bM_n(\IF)\to \bM_r(\IF)$ defined by
$$
A \mapsto \operatorname{trace}(A)E
$$
for any nonzero idempotent $E\in \bM_r(\IF)$.
For any $A,B\in \bM_n(\IF)$, observe that
\begin{multline*}
  \operatorname{trace}(AB + BA) = 2\operatorname{trace}(AB) =0 \\
= 2\operatorname{trace}(A)\operatorname{trace}(B)=\operatorname{trace}(A)\operatorname{trace}(B) + \operatorname{trace}(B)\operatorname{trace}(A).
\end{multline*}
Hence, $\Phi$ is a linear
Jordan homomorphism.
Since the trace of an idempotent in $\bM_n(\IF)$ is either zero or one (modulo $2$),
 $\Phi$ also sends idempotents to idempotents.
  But $\Phi$ does not assume the form \eqref{oldform-2}.
  Note also that the nonzero Jordan homomorphism $\Phi$ is not injective and the Jordan ideal $\Phi^{-1}(0)$ consisting of trace zero matrices is not a two sided ideal of $\bM_n(\IF)$.
Moreover,
it  does not send disjoint idempotents to disjoint idempotents,
and thus Theorem \ref{thm:main} does not apply.
\end{example}

\section{Linear preservers of matrices annihilated by a fixed polynomial}\label{S:4}

We  study a linear map
$\Phi: \bM_n(\IF)\to \bM_r(\IF)$  such that
\begin{equation}\label{poly}
f(\Phi(A)) = 0\quad\text{whenever}\quad f(A) = 0
\end{equation}
for a given $f(x) \in \IF[x]$ with distinct zeroes.
We will see that such linear maps $\Phi$ have nice structure.
In \cite{Hou}, the authors consider unital linear maps on complex operator algebras satisfying (\ref{poly}),
and show that $\Phi$ will preserve idempotents so that Corollary \ref{cor:J-homo} applies.
It turns out that the result can be extended to linear maps between matrices over
any field $\IF$, with  Theorem \ref{thm:main} instead, under some mild technical assumptions
on the zero set $Z(f)=\{a_1, \ldots, a_m\}$ of the polynomial $f(x)=(x-a_1)\cdots (x-a_m)$ with
$m\geq 2$ distinct zeroes.  Namely, we   assume that   $Z(f)$ is not a coset of an additive subgroup of $\IF$.
Equivalently, $\{a_j-a_1:1 \leq j \leq m\}$ is not an additive subgroup of
 $\IF$.
It also amounts to saying that the coset condition $a-b+c\in Z(f)$ does not always hold for any $a,b,c\in Z(f)$.

 Clearly, if $\IF$ has characteristic $0$, then the finite set $Z(f)$ cannot be a coset of an additive subgroup of $\IF$.
 It is also the case if the degree $m$ of $f$ is not a power of the characteristic $k$ of $\IF$ (see Remark \ref{rem:poly}).
We have examples below to illustrate that
 the technical assumption is necessary.
We also show with additional remarks  that characterizing those $\Phi$  satisfying
 (\ref{poly}) without the technical assumption would be challenging.

\subsection{Unital linear maps}

\begin{theorem}\label{thm:sep-poly}
 Let  $f(x) = (x-a_1)(x-a_2)(x-a_3) \cdots (x-a_m)$ with
$m\ge 2$ distinct zeroes
in a field $\IF$ such that
  its zero set $Z(f)$   is not an additive
coset.  Suppose that
 $\Phi: \bM_n(\IF) \rightarrow \bM_r(\IF)$ is a unital linear map.
The following conditions are equivalent.
\begin{itemize}
\item[{\rm (a)}] There are nonnegative integers $p, q$ satisfying
$(p+q)n = r$ and an invertible matrix $S \in \bM_r(\IF)$ such that
$\Phi$ assumes the form
\begin{align}\label{eq:poly-stand}
A \mapsto S^{-1}\begin{pmatrix} A\otimes I_p & \cr &  A^\T \otimes I_q\end{pmatrix}S.
\end{align}
\item[{\rm (b)}] $g(\Phi(A)) =0_r$ whenever $g(A) = 0_n$, for every polynomial $g(x) \in \IF[x]$.
\item[{\rm (c)}] $f(\Phi(A)) = 0_r$ whenever $f(A) = 0_n$.
\item[{\rm (d)}] $\Phi$ sends disjoint rank one idempotents to disjoint idempotents.
\end{itemize}
\end{theorem}

\it Proof. \rm
The implications (a) $\Rightarrow$ (b) $\Rightarrow$ (c) are clear.

Suppose (c) holds.
To derive (d), we first show that $\Phi$ sends idempotents to idempotents. We assume by contradiction that $\Phi(E)$ is not an idempotent for an idempotent $E$ in $\bM_n(\IF)$.
In this case, $E$ is neither zero nor the identity matrix.
Consider $g(x)=f(ax+b)$ for any $a,b\in \IF$.
Then condition (c) holds if and only if $g(\Phi(A))=0_r$ whenever $g(A)=0_n$.
Observe that if the zero set $Z(f)$ of $f$ is not   an additive coset,
then neither is the zero set $Z(g)$ of $g$.
Thus, by replacing  $f(x)$ with $g(x)$ for some suitable $a,b$, we can assume that $a_1=0$ and $a_2=1$,
and $Z(f)$ is not an additive group.

Note that $f(X) = 0$ if and only if the minimum polynomial of $X$ divides $f(x)$.
It amounts to saying that $X$ is diagonalizable
with eigenvalues from the set $Z(f) = \{0, 1, a_3, \dots, a_m\}$.
Since $f(a_j E) = f(a_j)E=0$ for all $j=1,\ldots, m$,
we have
 $f(a_j\Phi(E)) = 0$, and thus
$a_j\Phi(E)$ is diagonalizable
with eigenvalues in  $Z(f)$  for all $j=1,\ldots, m$.
If $\Phi(E)$ has any eigenvalue $\lambda\neq 0, 1$, then $\lambda Z(f) = Z(f)$ implies
$Z(f)$ contains $\lambda^l$ for all $l\in \IZ$.
On the other hand, $I_n - E$ is also a nonzero idempotent but not the identity.
By the same arguments, we see that
 $\Phi(I_n-E)=I_n-\Phi(E)$ has $1-\lambda$ as an eigenvalue, and thus $(1-\lambda)^l\in Z(f)$ for any $l\in \IZ$.

For any $a_i, a_j\in Z(f)$,  the matrix $a_i(I_n-E) + a_j E$ is  diagonalizable  with eigenvalues $a_i, a_j$, and
thus $f(a_i(I_n-E) + a_j E)=0$.
Consequently,
\begin{quote}
\textsc{(\dag)}\qquad all eigenvalues of $a_i(I_n-\Phi(E)) + a_j \Phi(E)$ are in $Z(f)$ for any $a_i, a_j\in Z(f)$.
\end{quote}
In particular, for $a_i = (1-\lambda)^{-1}$ and $a_j=\lambda^{-1}$, we have $2\in Z(f)$.
It follows from $Z(f)=\lambda Z(f) = (1-\lambda)Z(f)$ that $2\lambda^l, 2(1-\lambda)^l\in Z(f)$ for all $l\in \IZ$.
Looking at the eigenvalues of  $(2(1-\lambda)^{-1})(I_n-\Phi(E)) + (\lambda^{-1})\Phi(E)$,
we see that $3\in Z(f)$, as well as $3\lambda^l, 3(1-\lambda)^l\in Z(f)$ for all $l\in \IZ$.
Inductively, we see that $0,1,2, 3,4, \ldots$ are all in $Z(f)$.
We thus have a contradiction if $\IF$ has characteristic $0$, since the polynomial
$f(x)$ cannot have infinitely many distinct zeroes.

Assume now that $\IF$ has characteristic $k> 0$. Let $\IF_k=\{0,1,2,\ldots, k-1\}$ be the  prime subfield of $\IF$.
Recall that $Z(f)$ contains $0$, and all $i\lambda^j$ and $i(1-\lambda)^j$ with $i,j\in \mathbb{Z}$, and thus their
arbitrary finite sums by (\dag).
It amounts to saying that  $Z(f)$
contains the  subfield $\IF_k(\lambda)$ of $\IF$.
If $b_2$ is any zero of $f$ outside $\IF_k(\lambda)$, then by the above arguments, we see that $Z(f)$ contains
all $ib_2\lambda^j$ and $ib_2(1-\lambda)^j$ with $i,j\in \mathbb{Z}$, and thus their arbitrary finite  sums by (\dag).
In other words, $b_2\IF_k(\lambda)\subseteq Z(f)$.
It follows from (\dag) again that
$\IF_k(\lambda) + b_2\IF_k(\lambda)\subseteq Z(f)$.
We can work on other zeroes  of $f$ until we arrive at the conclusion that the $m$ element set
\begin{align*}
Z(f) = \IF_k(\lambda) + b_2\IF_k(\lambda) + b_3\IF_k(\lambda) + \cdots + b_j\IF_k(\lambda)
\end{align*}
 is an additive group, or equivalently, a vector space over $\IF_k(\lambda)$, as well as over $\IF_k$.
 This contradiction ensures that $\Phi$ sends idempotents to idempotents.

If $\IF$ does not have characteristic 2, being a linear idempotent preserver $\Phi$ sends disjoint idempotents to
disjoint idempotents.  Suppose $\IF$ has characteristic 2.  We can assume  that
$f(x)=x(x-1)(x-a_3)\cdots (x-a_m)$ for some distinct $a_1=0, a_2=1, a_3, \ldots, a_m\in \IF\setminus\IF_2$.
The assumption that $Z(f)$ not being an additive group also forces $m\geq 3$ in this case.
Let $E,F\in \bM_n(\IF)$ be nonzero disjoint
idempotents.  Since $E+F$ is also an idempotent,   all $\Phi(E)$, $\Phi(F)$ and
 $\Phi(E+F)=\Phi(E) + \Phi(F)$ are idempotents.
 We  see that $\Phi(E)\Phi(F)+\Phi(F)\Phi(E)=0$, and thus $\Phi(E)\Phi(F)=\Phi(F)\Phi(E)$.
It follows that the idempotents $\Phi(E),\Phi(F)$ are simultaneously diagonalizable.
We can thus assume both $\Phi(E), \Phi(F)$ are diagonal matrices with diagonal entries $0$ or $1$.
If $\Phi(E)\Phi(F)\neq 0$,
then we can further assume that the $(1,1)$ entries of both $\Phi(E),\Phi(F)$ are $1$.
Since $f(a_i F + a_j E)=0$, we have $f(a_i\Phi(F) + a_j\Phi(E))=0$, and especially
the $(1,1)$ entry of the diagonal matrix $a_i \Phi(F) + a_j \Phi(E)$ is a zero of $f(x)$,
which says
$a_i + a_j\in Z(f)$ for $i=1,\ldots,m$, and so do  all their finite sums (and differences).
It follows that $Z(f)$ is an additive group, a contradiction.
Hence, $\Phi$ sends disjoint idempotents to disjoint idempotents in any case.
This establishes (d)

Finally, the implication (d) $\Rightarrow$ (a) follows from
Theorem \ref{thm:main}, since  $\Phi$ is unital and, by the assumption, $\IF\neq \IZ_2$.
\qed

\begin{remark}\label{rem:poly}
(i) Suppose that $\IF$ is a  field of characteristic $k\neq 0$.
If $Z(f)-a_1$ is an additive group, then it is also a finite dimensional vector space over $\IF_k=\{0,1,\ldots, k-1\}$, and thus
$Z(f)$ has
$k^l$ elements  for some positive integers $l$.
In particular, when the degree $m$ of $f(x)$
 is not a power of $k$, we always have the implication (c) $\implies$ (d) in Theorem \ref{thm:sep-poly}.

(ii) We note that  counter examples to the conclusion of Theorem \ref{thm:sep-poly} are provided
in \cite{Hou}, when the polynomial $f(x)$ has
repeated zeroes.
\end{remark}

\begin{example}\label{eg:Full-AP}
Let a field $\IF$ have nonzero characteristic $k$, and denote by $\IF_k=\{0,1,2,\ldots, k-1\}$ the  prime subfield of $\IF$.
If the zero set $Z(f)$ of a polynomial $f(x)\in \IF[x]$ is an additive coset in $\IF$ then the equivalences in Theorem \ref{thm:sep-poly}
need not hold.

\smallskip

\noindent
(a) The unique polynomial $f(x)\in \IZ_2[x]$ with distinct zeroes in $\IZ_2$ is $f_{\IF_2}(x)=x(x-1)=x^2-x$.
Any linear map $\Phi: \bM_n(\mathbb{Z}_2)\to \bM_r(\IZ_2)$ preserving matrices annihilated by $f_{\IF_2}(x)$ is the one preserving
idempotents.
Consider the  map  defined by
$$
A\mapsto a_{11}I_r,
$$
where $a_{11}$ is the $(1,1)$ entry of $A\in \bM_n(\mathbb{Z}_2)$, and $a_{11}$ is either $0$ or $1$.
The above unital linear map sends idempotents to idempotents.
But it does not assume the form \eqref{eq:poly-stand}.

\smallskip

\noindent
(b) Suppose $k$ does not divide $n$.
Let $\Phi: \bM_n(\IF)\to \bM_r(\IF)$ be the unital linear map defined by
$\Phi(A) = n^{-1}\operatorname{trace}(A)I_r$.
Note that $\Phi(A)$ does not assume the form \eqref{eq:poly-stand}.

(i) Let
$f_{\IF_k}(x) = x(x-1)(x-2)\cdots (x-k+1)=x^k - x$ in $\IF[x]$.
It is plain that $f_{\IF_k}(A)=0$ for any matrix $A$ over $\IF$ exactly when $A$ is diagonalizable with eigenvalues in  $\IF_k$.
It follows that
$$
f_{\IF_k}(\Phi(A))=0\quad\text{whenever}\quad f_{\IF_k}(A)=0.
$$

(ii)
 More generally, let
$W$ be an additive coset of $\IF$;
in other words, $W-a$ is an additive group for any $a\in W$.
 Let $f_{W}(x)=\prod_{w\in W}\, (x-w)$.
It is plain that $f_{W}(A)=0$ for any matrix $A\in \bM_n(\IF)$ exactly when $A$ is diagonalizable with all $n$ eigenvalues in  $W$.
In this case, $n^{-1}\operatorname{trace}(A)\in W$.  Therefore,
we have again
$$
f_W(\Phi(A))=0_r \quad\text{whenever}\quad f_W(A)=0_n, \quad\forall A\in \bM_n(\IF).
$$

\smallskip

\noindent
(c) Suppose $k$ divides $n$.
In this case, the averaging trace functional is not defined.
However, if $\IF$ is a finite field consisting of
$k^l$ elements, then the polynomial $f(x)=\prod_{a\in \IF}\, (x-a) = x^{k^l}-x$
vanishes on $\IF$.  Thus the unital linear map $\Psi: \bM_n(\IF)\to \bM_r(\IF)$, defined
by $\Psi(A)= a_{11}E_{11} + a_{22}(I_r - E_{11})$ for any $A=(a_{ij})\in\bM_n(\IF)$, satisfies the condition that
$f(\Psi(A))=0$ for all $A\in \bM_n(\IF)$.
However, $\Psi(A)$ does not assume the form \eqref{eq:poly-stand}.
\end{example}

\begin{remark}\label{rem:poly-2}
  (a) With Example \ref{eg:Full-AP}, we see that the assumption that
$Z(f)$ not being an additive coset in Theorem \ref{thm:sep-poly}
is also necessary for the implication (c) $\implies$ (d),  when $n$ is not a multiple of the characteristic $k$ of $\IF$.

(b) We may also consider examples in which  the range space of $\Phi$
is a subspace of diagonal matrices or a subspace of diagonalizable matrices. For a field which contains
no $a$ such that $a^2=-1$ and the sum of two squares is always a square (namely,  formally real and Pythagorean),
all symmetric matrices are diagonalizable (see, e.g., \cite{MSV93}).
In this case, the map could be quite wild.
This further illustrates that solving the problem without the technical assumption is challenging.
\end{remark}

\subsection{Removal of the unital assumption}

In the following, we try to relax the assumption that $\Phi(I_n)=I_r$ in Theorem \ref{thm:sep-poly}.
Let $f(x)=(x-a_1)(x-a_2)\cdots (x-a_m)\in \IF[x]$ be a polynomial with distinct zeroes.
We call $\alpha\in\IF$ a multiplier of the zero set, or a \emph{zero multiplier}, of $f(x)$ if
$\alpha Z(f)\subseteq Z(f)$.  We note that  $\alpha Z(f)=Z(f)$ whenever $\alpha$ is a nonzero multiplier of $Z(f)$.

Let
\begin{align*}
M(f)
= \{\alpha\in \IF: \alpha Z(f) \subseteq Z(f)\}
\end{align*}
denote the set of zero multipliers of $f(x)$.
It is easy to see that  
$$
0\in M(f)\quad\text{exactly when}\quad 0\in Z(f).
$$
Since $M(f)\setminus \{0\}$ is a finite multiplicative subgroup of $\IF\setminus\{0\}$, there is a least positive integer $h$
such that $M(f)\setminus \{0\}$ is the cyclic group $\{\lambda, \ldots, \lambda^h =1\}$ for a primitive $h$th root
$\lambda$ of unity in $\IF$.
Let $b_1\in Z(f)\setminus\{0\}$.
We have  $\lambda b_1, \lambda^2 b_1, \ldots, \lambda^h b_1=b_1$ in $Z(f)$, and thus
 $h\leq m$.  If there is any nonzero element $b_2$ in $Z(f)$
lying outside this cycle, then we have another cycle $\lambda b_2, \lambda^2 b_2, \ldots, \lambda^h b_2=b_2$ in $Z(f)\setminus\{0\}$.
In finite steps, we can partition $Z(f)\setminus\{0\}$ into cycles in this form; namely,
$$
Z(f)\setminus \{0\} = \bigcup_j \{b_j, \lambda b_j, \ldots, \lambda^{h-1} b_j\}
= \big(\bigcup_{j} b_jM(f)\big)\setminus \{0\}.
$$
We conclude that either $h$ divides $m$ and thus $\lambda^m=1$ when $0\notin Z(f)$, or  $h$ divides $m-1$ and thus
$\lambda^m=\lambda$ when  $0\in Z(f)$.
Accordingly, for any $\alpha\in M(f)\setminus\{0\}$, we have
\begin{align*}
f(\alpha x) &= f(x)\quad  \text{when $0\notin Z(f)$,}\\
 f(\alpha x) &= \alpha f(x)\quad   \text{when $0\in Z(f)$,\quad and}\\
Z(f(x))& =Z(f(\alpha x)) \quad  \text{in both cases.}
\end{align*}

\begin{theorem}\label{thm:without1}
Let $f(x)=(x-a_1)(x-a_2)\cdots (x-a_m)\in \IF[x]$  with $m\geq 2$ distinct zeroes in a field $\IF$
such that the zero set $Z(f)$ is not an additive coset.
Let  $\Phi: \bM_n(\IF)\to \bM_r(\IF)$ be a linear map  preserving matrices annihilated by $f(x)$.
Suppose that $\Phi(I_n)$ commutes with $\Phi(A)$ for all $A\in \bM_n(\IF)$.
Then there are nonnegative integers $p, q$ with
$s=r-np-nq\geq 0$, invertible diagonal matrices $D_1\in \bM_p(\IF)$, $D_2\in \bM_q(\IF)$, and an invertible matrix
$S  \in \bM_{r}(\IF)$ such that
\begin{align}\label{eq:Phi-canonical}
\Phi(A) =  S^{-1}\begin{pmatrix} A\otimes D_1 & &\cr &  A^\T \otimes D_2 &\cr & & 0_s\cr \end{pmatrix}S,
\quad\text{for all $A\in \bM_n(\IF)$}.
\end{align}
where $ D_1$ and $D_2$ have diagonal entries  from  $M(f)$, and $s=0$ if $0\notin Z(f)$.
\end{theorem}
\begin{proof}
For any nonzero $a_j\in Z(f)$, we have $f(a_jI_n)=0$, and thus $f(a_j\Phi(I_n))=0$.
It follows $\Phi(I_n)$ is diagonalizable.
If $\lambda$ is an eigenvalue of $\Phi(I_n)$, then $f(\lambda a_j)=0$ for all $j=1,\ldots, m$.
This says $\lambda\in M(f)$.
Since $\Phi(I_n)\Phi(A)=\Phi(A)\Phi(I_n)$ for all $A\in \bM_n(\IF)$,
after a similarity transformation, we can assume that
$$
\Phi(I_n) =
\begin{pmatrix}
\lambda_1 I_{r_1} & & &\cr
& \lambda_2 I_{r_2} & & \cr
& & \ddots && \cr
& & & \lambda_l I_{r_l}
\end{pmatrix},
$$
with distinct eigenvalues $\lambda_1, \lambda_2, \ldots, \lambda_l$ from $M(f)$,
and there are linear maps $\Phi_j:\bM_n(\IF)\to \bM_{r_j}(\IF)$ with $\Phi_j(I_n)=\lambda_j I_{r_j}$ for $j=1,2,\ldots, l$ such that
$$
\Phi(A) = \begin{pmatrix}
\Phi_1(A) & & &\cr
& \Phi_2(A) & & \cr
& & \ddots && \cr
& & & \Phi_l(A)\end{pmatrix} \quad\text{for all $A\in \bM_n(\IF)$}.
$$

In case when some $\lambda_j=0$,  we claim that $\Phi_j$ is the zero map.
We note that  $0\in Z(f)$ since $0\in M(f)$.
Let $E$ be any  idempotent in $\bM_n(\IF)$. Since $\Phi_j(E) + \Phi_j(I_n-E)= \Phi_j(I_n) =0$, we
have $\Phi_j(E)=-\Phi_j(I_n-E)$.  For any $a,b\in Z(f)$, it follows from $f(aE + b(I_n-E)) = 0$ that
$f(a\Phi_j(E)+b\Phi_j(I_n-E))=f((a-b)\Phi_j(E))=0$.  If $\alpha$ is any nonzero eigenvalue of the diagonalizable matrix
$\Phi_j(E)$, we have $\alpha (a-b)\in Z(f)$ for all $a,b\in Z(f)$.
Letting $b=0$, we see that  $Z(f)=\alpha Z(f)$, which is not an additive group.
Therefore, there  are $\alpha a,\alpha b\in Z(f)$ such that $\alpha(a-b)\notin Z(f)$.
This contradiction says that $\Phi_j(E)$ has no nonzero eigenvalue, and thus $\Phi_j(E)=0$ for all idempotents $E$ in $\bM_n(\IF)$.
This forces $\Phi_j=0$ by Lemma \ref{P1P2} (a).

For those nonzero zero multiplier $\lambda_j$ of $f(x)$, it follows from the discussion before Theorem \ref{thm:without1}
that their inverses $\lambda_j^{-1}\in M(f)$, and the
unital linear map ${\lambda_j}^{-1}\Phi_j$ preserves matrices annihilated by $f(x)$.
We can then apply Theorem \ref{thm:sep-poly} to establish the desired assertions.
\end{proof}

The condition that $\Phi(I_n)$ commutes with all elements
in the range space of $\Phi$ may follow from other assumptions.

\begin{theorem}\label{thm:whenPhiunital}
  Let
$f(x)=(x-a_1)(x-a_2)\cdots (x-a_m)\in \IF[x]$ with $m\geq 2$ distinct zeroes in a field $\IF$
such that the zero set $Z(f)$ is not an additive coset.
 Suppose that  $\Phi: \bM_n(\IF)\to \bM_r(\IF)$ is a linear map preserving matrices annihilated by $f(x)$.
 If  $f(0)=0$ or $M(f)= \{1\}$,
  then $\Phi$ assumes the form \eqref{eq:Phi-canonical}.
\end{theorem}
\begin{proof}
Note first that as seen in the proof of Theorem \ref{thm:without1},  $\Phi(I_n)$ is diagonalizable with eigenvalues from $M(f)$.
If $M(f)=\{1\}$ then $\Phi(I_n)=I_r$ and  Theorem \ref{thm:sep-poly} applies.

Suppose  $f(0)=0$, and  assume that $Z(f)$ is not an additive group.
The aim is to show that $\Phi(I_n)$ commutes with all elements
in the range space of $\Phi$.

Let $h$ be the order of the cyclic group  $M(f)\setminus \{0\}= \{\lambda, \ldots, \lambda^h =1\}$.
We claim that $h^{-1}\in \IF$.
It is clear the case when $\IF$ has characteristic $k=0$.
Suppose $\IF$ have characteristic $k>0$ instead.
Since    $M(f)\setminus \{0\}$ can be considered as
a multiplicative  subgroup of  $\IF_k[\lambda]\setminus\{0\}$, which has $k^l -1$ elements for
some positive integer $l$.  Thus $h$ divides $k^l-1$.
Since $k$ and $k^l-1$ are coprime, we see that  $h$ is not a multiple of $k$, and thus $h^{-1}$ exists in $\IF$.

For any nontrivial idempotent $E$ of $\bM_n(\IF)$, any $a\in Z(f)$ and any $t=1,\ldots, h$,
we have $f(a( E + \lambda^t (I_n-E)))=0$, and thus $f(a(\Phi(E) + \lambda^t \Phi(I_n - E)))=0$.
It follows that $\Phi(E) + \lambda^t \Phi(I_n-E)$ is diagonalizable with eigenvalues from $M(f)$.
In particular,
$$
(\Phi(E) + \lambda^t \Phi(I_n-E))^{h+1} = \Phi(E) + \lambda^t \Phi(I_n-E), \quad\text{for $t=1, \ldots, h$.}
$$
Similar reasonings ensure that
$\Phi(E)^{h+1}=\Phi(E)$ and $\Phi(I_n-E)^{h+1}=\Phi(I_n-E)$.

In the following, we imitate the proof of  \cite[Lemma 2.2]{Hou} which deals with the complex case.
Denote by $E'=\Phi(E)$ and $F'=\Phi(I_n-E)$.
Then  ${E'}^{h+1}={E'}$, ${F'}^{h+1}={F'}$ and
\begin{align}\label{eq:EF}
  ({E'} + \lambda^t {F'})^{h+1} = {E'} + \lambda^t {F'}, \quad\text{for $t=1,\ldots, h$.}
\end{align}
  Let
  $$
  q_j=\ \text{the sum of all noncommutative products of ${h+1-j}$ many ${E'}$'s and $j$ many ${F'}$'s,}
  $$
for $j=1, \ldots, h$.
  Expanding \eqref{eq:EF}, we have
  $$
  {E'}^{h+1} + \sum_{j=1}^{h} \lambda^{jt}q_j + \lambda^{t(h+1)}{F'}^{h+1}= {E'} + \lambda^t {F'},
  $$
or
$$
\sum_{j=1}^{h} \lambda^{jt} q_j = 0, \quad\text{for all $t=1,\ldots, h$.}
$$
Since $\sum_{t=1}^{h} \lambda^{jt} = \lambda^j(1-\lambda^j)^{-1}(1-\lambda^{jh})=0$ for $j=1,\ldots, h-1$, we have
$$
\sum_{t=1}^{h} \sum_{j=1}^{h} \lambda^{jt} q_j = \sum_{j=1}^{h} q_j \sum_{t=1}^{h} \lambda^{jt} = hq_h =0.
$$
Since $h^{-1}$ exists in  $\IF$, we have $q_h=0$.
It follows from
\begin{align*}
q_h{F'} &= ({E'}{F'}^h + {F'}{E'}{F'}^{h-1} +\cdots + {F'}^h{E'}){F'}=0 \\
= {F'}q_h &= {F'}({E'}{F'}^h + {F'}{E'}{F'}^{h-1} +\cdots + {F'}^h{E'})
\end{align*}
that
$$
{E'}{F'} = - {F'}{E'}{F'}^{h} -\cdots - {F'}^h{E'}{F'} = {F'}{E'}.
$$
Consequently, $\Phi(E)\Phi(I_n) = E'(E'+F') = (E'+F')E' = \Phi(E)\Phi(I_n)$ for every idempotent $E\in \bM_n(\IF)$, and thus
$\Phi(I_n)$ commutes with all $A\in \bM_n(\IF)$.  Then Theorem \ref{thm:without1} applies.
\end{proof}

The following examples demonstrate  that the situation can be difficult if $f(0)\neq 0$ and $\Phi(I_n)$ does not commute with all $\Phi(A)$.

\begin{example}\label{eg:involution}\label{eg:counter-involutions}
Let $\IF$ be an arbitrary field, and $f(x)=(x-a_1)(x-a_2)\in \IF[x]$ with $a_1\neq a_2$.
 Suppose that  $\Phi: \bM_n(\IF)\to \bM_r(\IF)$ is a linear map preserving matrices annihilated by $f(x)$.

(a) Suppose that $a_1\neq - a_2$.  The zero set $Z(f)=\{a_1, a_2\}$ is not an additive coset
exactly when $\IF$ does not have characteristic $2$.
In this case, either $f(0)=0$ or the zero multiplier set $M(f)= \{1\}$.
 Theorem \ref{thm:whenPhiunital} ensures
that $\Phi$ assumes the form 
$$
A\mapsto S^{-1}\left[(A\otimes I_p)\oplus (A^\T\otimes I_q) \oplus 0_s\right]S,
$$
for an invertible $S\in \bM_r(\IF)$, and nonnegative integers $p,q,s$ such that $np+nq+s =r$.

If $\IF$ has characteristic $2$, then $\Phi$ might not assume the above form as Example \ref{eg:Full-AP}(a) demonstrates.

(b)
Suppose that $a_1=-a_2\neq 0$.  It can be reduced to that
 $f(x)=(x-1)(x+1) = x^2 -1$.
In this case,  both $\pm 1$ are zero multipliers of $f(x)$.
We see that
$\Phi$ preserves
involutions, that is, $\Phi(A)^2= I_r$ if $A^2=I_n$, but $\Phi$ does not always assume the expected form.
For example, consider the   map  $\Phi: \bM_2(\IF)\to \bM_4(\IF)$ sending
\begin{align}\label{eq:M2-M4}
\left(      \begin{array}{cc}        a & b \\        c & d \\      \end{array}    \right)\quad\text{to}\quad
\left(       \begin{array}{cccc}         a & b &&\\  c & -a && \\ && a&c\\ && b& d       \end{array}     \right).
\end{align}
Then $\Phi$ is  linear, injective  and sends symmetric matrices to symmetric matrices.  Moreover,
$\Phi(A)$ is an involution if and only if $A$ is an involution.
However, $\Phi(I_2)$ does not commute with  all $\Phi(A)$, and $\Phi$ does not assume   the   form \eqref{eq:Phi-canonical}.

In general, observe that
 $V \in \bM_2(\IF)$ is an involution if and only if $V = \pm I_2$ or $V$ has eigenvalues $1,-1$.
On the other hand, any $A\in \bM_2(\IF)$ can be written as
$A=(A- \operatorname{trace}(A)I_n) +\operatorname{trace}(A)I_2$ as a linear sum of a trace zero element and the identity matrix $I_2$.
One may define $\Phi$ by linearity such that $\Phi$ sends every trace zero $A$ to
$ S^{-1}[(A\otimes D_1)\oplus (A^\T \otimes D_2)\oplus 0]S \in \bM_r(\IF)$, and assign $\Phi(I_2) = D \in \bM_r(\IF)$, for any
 involutions $D_1, D_2, D$,  and invertible matrix $S$ (of appropriate sizes).
 Then $\Phi$ will preserve involutions, but $\Phi$ might not assume
  the form \eqref{eq:Phi-canonical}.
When we put $D_1=D_2=I_2$ and $D = I_4- 2E_{22}$, we get the example given in \eqref{eq:M2-M4}.
  For $n \geq 3$, it is however unclear whether such examples exist.
\end{example}

\section{Some possible extensions}\label{S:future}

As mentioned before, one may consider a further extension of Theorem \ref{thm:sep-poly} without any \textit{a priori} condition on $\Phi(I_n)$ or $M(f)$.
Some simple examples of $f(x)$ not covered by our results is $f(x) = x^m-1$ for $m\geq 3$.
Another interesting problem is characterizing linear maps $\Phi$ satisfying (\ref{poly}) for a general polynomial, which cannot be factorized as linear factors, or with repeated zeroes, in Theorems \ref{thm:sep-poly}, \ref{thm:without1} and \ref{thm:whenPhiunital}.

One may also extend our results to
$\Phi: \bM_n(\IF) \rightarrow \bL(Y)$, where $\bL(Y)$ is the set of all
linear operators on a linear space $Y$ over $\IF$.
If we know that $\Phi$ sends disjoint rank one idempotents to
disjoint idempotents, we can show that there is invertible $S \in B(Y)$
such that Lemma \ref{Ejj-form} holds. Then we can follow the proof of Theorem
\ref{thm:main} to modify $S$ and conclude that $\Phi(E_{ij}) = E_{ij}\otimes I_{K}$
for a subspace $K$ of $Y$.

 It would be interesting to consider linear maps
$\Phi: {\mathcal A} \rightarrow {\mathcal B}$,
where ${\mathcal A}$ and ${\mathcal B}$ are subspaces of
operators acting on linear spaces $X$ and $Y$ over $\IF$, respectively;
see, e.g., \cite{ZTC06,Kuzma}.  Note that the previous
studies often assume that the characteristics of
$\IF$ is not 2. Some of our techniques may be used to relax this condition,
and just assume that $\IF \ne \IZ_2$. For example, suppose
${\mathcal A}$ contains all finite rank operators in $\bL(X)$ and the identity operator $I$
and ${\mathcal B} = \bL(Y)$. If $\Phi: {\mathcal A} \rightarrow B(Y)$
preserves disjoint rank one idempotents,
then for any $n^2$-dimensional subspace  $\bV_n(T)$ of ${\mathcal A}$
with operators of the form
$T^{-1}(A \oplus 0)T$ for an invertible operator $T$, we can apply our results
to show that the restriction of $\Phi$ on $\bV_n(T)$ will be of the form
\eqref{oldform-2} for some invertible $S$, where $S$ may depend on $T$.
If we have some additional assumptions on $\Phi$, we may be able to show that
the operator $S$ is independent of $T$.

One may also consider other operator algebras ${\mathcal A}$
such as the algebra of  upper triangular matrices, or nested algebras, etc.
Also, it is interesting to extend our results or proof techniques to
additive maps, or multiplicative maps.

\section*{Acknowledgment}

C.-K. Li is an affiliate member of the Institute for Quantum Computing,
University of Waterloo. His research is supported by
Simons Foundation Grant 851334. This
research
started during his academic visit to Taiwan in 2018, which was supported by grants from Taiwan MOST.
He would like to express
his gratitude to the hospitality of several institutions there, including the
Academia Sinica,
National Chung Hsing University,  National Sun
Yat-sen University, and National Taipei University of Science and Technology.

M.-C. Tsai, Y.-S. Wang and N.-C. Wong are supported by Taiwan MOST grants 110-2115-M-027-002-MY2,
111-2115-M-005-001-MY2 and  110-2115-M-110-002-MY2,
respectively.


\begin{thebibliography}{WWW}



\bibitem{Bai}
Z. Bai and J. Hou, Linear maps and additive maps that preserve
operators annihilated by a polynomial,
J. Math.\ Anal.\ Appl., \textbf{271} (2002), 139--154.


\bibitem{CKLW03}
M. A. Chebotar, W.-F. Ke, P.-H. Lee and N.-C. Wong,
Mappings preserving zero products, \emph{Studia Math.}, \textbf{155}:1 (2003), 77--94.







\bibitem{FKKS01}
A. Fo\v{s}ner , B. Kuzma , T. Kuzma and N.-S. Sze,  Maps preserving
matrix pairs with zero Jordan product,
\emph{Linear and Multilinear Algebra}, \textbf{59}:5 (2011), 507--529.


\bibitem{GLS} A. Guterman, C.-K. Li, and P. \v{S}emrl,
Some general techniques on linear preserver problems,
\emph{Linear Algebra Appl.}, \textbf{315} (2000), 61--81.







\bibitem{Hou}
J. Hou  and S. Hou, Linear maps on operator algebras that preserve elements
annihilated by a polynomial,
\emph{Proc.\ Amer.\ Math.\ Soc.}, \textbf{130} (2002), 2383--2395.


\bibitem{Howard} R. Howard,
Linear maps that preserve matrices annihilated by a polynomial, \emph{Linear Algebra Appl.},
{\bf  30} (1980), 167--176.




\bibitem{Kuzma} B. Kuzma,
Additive idempotent preservers, {\it Linear Algebra Appl.},
{\bf  355} (2002), 103--117.




\bibitem{LP} C.-K. Li and S. Pierce,
Linear preserver problems, \emph{Amer.\ Math.\ Monthly}, \textbf{108} (2001), 591--605.





\bibitem{LTWW20}
C.-K. Li, M.-C. Tsai, Y.-S. Wang, and N.-C. Wong,
Nonsurjective maps between rectangular matrix spaces
preserving disjointness, triple products, or norms,
\emph{J. Operator Theory}, \textbf{83} (2020),   27--53.

\bibitem{LT92}
C.-K. Li and N.-K. Tsing,
Linear preserver problems: A brief introduction and some special techniques,
\emph{Linear Algebra Appl.}, \textbf{162-164} (1992),  217--235.

\bibitem{LCLW18}
J.-H. Liu, C.-Y. Chou, C.-J. Liao and N.-C. Wong,
Linear disjointness preservers of operator algebras and related structures, \emph{Acta Sci.\ Math.\ (Szeged)},
 \textbf{84} (2018),  277--307.


\bibitem{Monlar} L. Monlar,
\emph{Selected preserver problems on algebraic structures of linear operators and on function spaces},
Springer-Verlag, Berlin, 2007.


\bibitem{MSV93}
D. Mornhinweg, D. B. Shapiro and K. G. Valente,
The principal axis theorem over arbitrary fields,
\emph{Amer.\ Math.\ Monthly}, \textbf{100}:8 (1993), 749--754.



\bibitem{Semrl96}
P. \v{S}emrl,
Linear mappings that preserve operators annihilated by a polynomial,
\emph{J. Operator Theory}, \textbf{36} (1996), 45--58.


\bibitem{Semrl06}
P. \v{S}emrl,
Maps on matrix spaces, \emph{Linear Algebra Appl.}, \textbf{413} (2006), 364--393.






\bibitem{Wong05}
 N.-C. Wong,
 Triple homomorphisms of operators algebras,
 \emph{Southeast Asian Bull.\ Math.}, \textbf{29} (2005), 401--407.

\bibitem{W-zpp}
  N.-C. Wong,
  Zero product preservers of $C^*$-algebras,
 \emph{Contemp.\ Math.}, \textbf{435} (2007), 377--380. (arXiv:0708.3718)



\bibitem{ZTC06}
X. Zhang, X.-M. Tang and C.-G. Cao, \emph{Preserver problems on spaces of matrices}, Science Press, Beijing, 2006.


\end{thebibliography}
\end{document}